\documentclass[reqno,centertags,12pt]{amsart}
\usepackage{amsmath,amsthm,amscd,amssymb,latexsym,verbatim}

\usepackage{graphicx,epsf,cite}

-\marginparwidth \oddsidemargin -4mm \evensidemargin \oddsidemargin

\newtheorem{theorem}{Theorem}[section]

\newtheorem{lemma}[theorem]{Lemma}

\theoremstyle{definition}

\theoremstyle{remark}

\newcounter{smalllist}

\DeclareMathOperator*{\dist}{dist}

\DeclareMathOperator*{\sgn}{sgn} 
\DeclareMathOperator*{\ches}{ches}
\DeclareMathOperator*{\ch}{ch}
\allowdisplaybreaks
\numberwithin{equation}{section}

\newcommand{\lb}{\label}

\newcommand{\supp}{\text{\rm{supp}}}

\newcommand{\beq}{\begin{equation}}
\newcommand{\eeq}{\end{equation}}

\newcommand{\bal}{\begin{align}}
\newcommand{\eal}{\end{align}}
\newcommand{\bals}{\begin{align*}}
\newcommand{\eals}{\end{align*}}

\newcommand{\bbR}{{\mathbb{R}}}

\newcommand{\bbP}{{\mathbb{P}}}
\newcommand{\bbE}{{\mathbb{E}}}

\newcommand{\bbT}{{\mathbb{T}}}

\newcommand{\calL}{{\mathcal L}}
\newcommand{\calC}{{\mathcal C}}

\newcommand{\eps}{\varepsilon}

\newcommand{\tht}{\theta}

\newcommand{\laa}{\lambda}

\newcommand{\til}{\tilde}

\newcommand{\aaa }{a}

\begin{document}
\title[KPP Transition Fronts]
{Transition Fronts in Inhomogeneous Fisher-KPP Reaction-Diffusion Equations}

\author{Andrej Zlato\v s}

\address{\noindent Department of Mathematics \\ University of
Wisconsin \\ Madison, WI 53706, USA \newline Email: \tt
zlatos@math.wisc.edu}


\begin{abstract}
We use a new method in the study of Fisher-KPP reaction-diffusion equations to prove existence of transition fronts for inhomogeneous KPP-type non-linearities in one spatial dimension.  We also obtain new estimates on entire solutions of some KPP reaction-diffusion equations in several spatial dimensions.  Our method is based on the construction of sub- and super-solutions to the non-linear PDE from solutions of its linearization at zero.
\end{abstract}

\maketitle

\section{Introduction and Main Results} \lb{S1}

We introduce a new elementary method for the study of certain solutions to reaction-diffusion equations with  {\it Kolmogorov-Petrovskii-Piskunov (KPP) type} non-linearities.  We use it to prove existence of transition front solutions for very general spatially inhomogeneous KPP reaction-diffusion equations in one dimension as well as some special ones in several dimensions, and to obtain very good estimates on these solutions.  Our method is based on relating the solutions of the original non-linear equation to those of its linearization at $u=0$.

Let us first consider the reaction-diffusion equation
\beq \lb{1.1}
u_t=u_{xx} + f(x,u)
\eeq
with $x\in \bbR$ and $f$ an {\it inhomogeneous KPP reaction function}.
That is, we assume that $f$ is Lipschitz, $a(x)\equiv f_u(x,0)>0$ exists, 
\beq \lb{1.2}
f(x,0)=f(x,1)=0 \qquad\text{and} \qquad  a(x)g(u)\le  f(x,u)\le a(x)u \quad \text{for $(x,u)\in\bbR\times[0,1]$,}
\eeq
where $g\in C^1([0,1])$ is such that 
\beq \lb{1.3}
g(0)=g(1)=0,  \qquad g'(0)=1,  \qquad  \text{and} \qquad 0<g(u)\le u \quad \text{for $u\in(0,1)$.}
\eeq
We will also assume
\beq \lb{1.3a}
\int_0^1 \frac{u-g(u)}{u^2} du<\infty \qquad\text{and}\qquad g'(u)\le 1 \quad \text{for $u\in(0,1)$.}
\eeq
We define $a_-\equiv \inf_{x\in\bbR} a(x)\ge 0$  and also assume existence of $a_+<\infty$ such that
\beq \lb{1.4}
a(x)\le a_+ \qquad \text{for  $x\in\bbR$}.
\eeq

A (right-moving)  
{\it transition front} for \eqref{1.1} is an {\it entire} (global-in-time) solution $0\le u\le 1$ connecting 0 and 1 in the sense of
\beq \lb{1.6}
\lim_{x\to-\infty}u(t,x)=1 \qquad\text{and}\qquad \lim_{x\to+\infty}u(t,x)=0
\eeq
for each $t\in\bbR$.  It models an {\it invasion} of the unstable state $u\equiv 0$ by the asymptotically stable state $u\equiv 1$.  Moreover, we also require that for any $\eps>0$ there is $L_\eps<\infty$ such that 
\beq \lb{1.7}
 \sup_{t\in\bbR} {\rm diam}\left\{x\in\bbR \,|\,\eps\le u(t,x)\le 1-\eps\right\} \le L_\eps,
\eeq
that is, the width of the transition region between $\eps$ and $1-\eps$ is uniformly bounded in time.  This definition of transition fronts has first appeared in \cite{BH2,Matano}.

It has been well known since the seminal works of Fisher \cite{Fisher} and Kolmogorov-Petrovskii-Piskunov \cite{KPP} that in the homogeneous case $f(x,u)=f(u)$, there exist transition fronts where constant-in-time speed and profile.  More specifically, \eqref{1.1} has solutions of the form $u(t,x)=U(x-ct)$ with $U(-\infty)=1$ and $U(\infty)=0$ precisely when the front speed $c\ge c^*_f$, with $c^*_f\equiv 2\sqrt{f'(0)}$ is the {\it minimal front speed}.  These fronts have a constant-in-time profile $U$ with $U'<0$, are unique for each $c$ up to a translation, and are usually called {\it traveling fronts}.  There are also other transition fronts in this case \cite{HN1}, which are obtained as a combination of two or more traveling fronts with different speeds (we will discuss this in more detail below).  Later, existence of KPP transition fronts with time-periodic profiles (called {\it pulsating fronts}) was proved for $x$-periodic reactions $f$, again for all speeds $c\ge c^*_f$ with some $c^*_f>0$ \cite{BH}.

Very recently, existence of transition fronts was first time proved for some non-periodic inhomogeneous KPP reactions \cite{NRRZ}  (see  \cite{MNRR,MRS,NolRyz,ZlaGenfronts} for results on ignition reactions, and  \cite{ZlaGenfronts} for results on some non-KPP non-negative reactions). Specifically, if $a_->0$ and $a(x)-a_-$ is {\it compactly supported}, then transition fronts exist when  $ \lambda_0\equiv \sup \sigma[\partial^2_{xx}+a(x)]$, the supremum of the spectrum of the operator $\partial^2_{xx}+a(x)$, satisfies $ \lambda_0<2a_-$ (note that always $\lambda_0\ge a_-$).  These fronts do not have a constant profile but for each $c\in(2\sqrt{a_-}, \lambda_0 (\lambda_0 -a_-)^{-1/2})$ there is a front which has a {\it mean speed} 
\beq \lb{1.7a}
\lim_{|t-s|\to\infty} \frac{X(t)-X(s)}{t-s}
\eeq
equal to $c$, where $X(t)$ is the rightmost point such that $u(t,X(t))=\tfrac 12$.  Moreover, no transition fronts exist when, in addition, $a(x)\ge a_-$ and $ \lambda_0>2a_-$ \cite{NRRZ}; this is the first {\it non-existence-of-fronts} result.

We consider here the question of existence of transition fronts in general inhomogeneous media without the assumption of compact support of $a(x)-a_-$ (in which case no constant or mean speed fronts exist in general) and answer it in the affirmative again when $ \lambda_0<2a_-$.  We achieve this by using a new and elementary method which exploits the close connection between the equation \eqref{1.1} and its linearization 
\beq \lb{1.8}
v_t=v_{xx} + a(x)v
\eeq
at $u=0$.  

Such a connection is well known,  in particular, when $f(x,u)=f(u)$ and so $a(x)\equiv a=f'(0)$ is constant.  Then \eqref{1.8} has traveling-front-like solutions $e^{-\gamma(x-c_{a,\gamma}t)}$ with $\gamma>0$ and speed $c_{a,\gamma}\equiv \gamma+a\gamma^{-1}\ge 2 \sqrt a=c_f^*$.
It turns out \cite{Uchi} that if $c>2 \sqrt a$ and $\gamma<\sqrt a$ is such that $c=c_{a,\gamma}$, then the traveling front for \eqref{1.1} with speed $c$ also has asymptotic decay $e^{-\gamma(x-c_{a,\gamma}t)}$ as $x\to\infty$, while for $c=2 \sqrt a$, the asymptotic decay is $(x-2\sqrt a\, t)e^{-\sqrt a(x-2\sqrt a\, t)}$ as $x\to\infty$ (fronts for \eqref{1.8} with $\gamma>\sqrt a$ do not give rise to fronts for \eqref{1.1}).  
This means that if $U_{f,\gamma}$ is a traveling front profile for \eqref{1.1} corresponding to speed $c_{a,\gamma} \ge c_f^*$ with $\gamma\le \sqrt a$, and the function $h:[0,\infty)\to[0,1)$ is given by $U_{f,\gamma}(x)=h(e^{-\gamma x})$ (so that $h(0)=0$ and $\lim_{v\to\infty} h(v)=1$), then $h'(0)=1$ when $\gamma< \sqrt a$ and $\lim_{v\to 0}h(v)(-v\ln v)^{-1}=1$ when $\gamma= \sqrt a$, after an appropriate translation of $U_{f,\gamma}$ in $x$.

The above shows that for $f(x,u)=f(u)$ and for faster-than-minimal speed $c>c^*_f$, the ``tails'' of the corresponding traveling fronts for \eqref{1.1} and  \eqref{1.8} are asymptotically the same.  We will show that this still holds for some transition fronts in general inhomogeneous media when $ \lambda_0<2a_-$.  We will in fact show that {\it the study of these fronts for \eqref{1.1} is essentially equivalent to the study of the corresponding front-like solutions for the simpler equation \eqref{1.8}}.  

Similarly to the compactly supported $a(x)-a_-$ setting in \cite{NRRZ}, examples of the latter can be found in the form $v_{\laa}(t,x)\equiv e^{\laa t}\phi_\laa(x)$, where $\phi_\laa(x)>0$ is a 
solution of the Schr\" odinger generalized eigenfunction equation
\[
\phi_\laa''+a(x)\phi_\laa=\laa\phi_\laa,
\]
with $\lim_{x\to\infty} \phi_\laa(x)=0$ and $\phi_\laa(0)=1$.  Notice that if $a$ is constant, then $v_\laa(t,x)=e^{\laa t-\sqrt{\laa-a}\,x}=e^{-\gamma(x-c_{a,\gamma}t)}$ with $\gamma\equiv \sqrt{\laa-a}$.

Sturm oscillation theory shows that such $\phi_\lambda>0$ exists and is unique precisely when $\laa> \lambda_0$.  Moreover, $\phi_\lambda$ grows exponentially as $x\to -\infty$ (see \eqref{2.8}).  Then $v_{\laa}$ is a super-solution of \eqref{1.1} and we will show that for any $\laa\in(\lambda_0,2a_-)$ there is $h:[0,\infty)\to[0,1)$ such that $w_{\laa}(t,x)\equiv h(v_{\laa}(t,x))$ is a sub-solution (rather than an outright solution, as in the homogeneous case).  Moreover, $\lambda<2a_-$ will ensure $h(v)\le v$  so it will follow that there exists a transition front $u\in[w_{\laa}, v_{\laa}]$ for \eqref{1.1}.  We note that this construction cannot be expected to work for  $\lambda\ge 2a_-$ in general because in the homogeneous case this translates to $\gamma\ge\sqrt a$, which either gives rise to no front for \eqref{1.1} when $\gamma>\sqrt a$ or violates $h(v)\le v$ when $\gamma = \sqrt a$.


There is, in fact, a larger class of positive entire solutions of \eqref{1.8}, of which the $v_\lambda$ are the extremal points.  Indeed, if $\mu$ is a finite non-negative non-zero Borel measure on $(\lambda_0,\infty)$ with a bounded support, then Harnack inequality shows that
\beq \lb{1.8b}
v_\mu(t,x)\equiv \int_\bbR  v_{\laa}(t,x) d\mu(\laa) = \int_\bbR e^{\laa t} \phi_{\laa}(x) d\mu(\laa)
\eeq
is well-defined, and it is obiously an entire solution of \eqref{1.8}.  We will show that $v_\mu$ also gives rise to an entire solution of \eqref{1.1} provided $ \sup \supp(\mu)<2a_-$.

Finally, our result extends to and will be stated for the more general PDEs
\beq \lb{1.9}
u_t=(B(x)u_{x})_{x} + q(x) u_x + f(x,u)
\eeq
and
\beq \lb{1.10}
v_t=(B(x)v_{x})_{x} +q(x) v_x + a(x)v
\eeq
with $B,q$ Lipschitz and satisfying
\beq \lb{1.11}
0<B_-\le B(x)\le B_+<\infty \qquad \text{and} \qquad  |q(x)|\le q_+<\infty \quad \text{\,\,for  $x\in\bbR$}.
\eeq
Let us define
\beq \lb{1.11a}
\lambda_0\equiv \sup_{\psi \in H^1(\bbR)} \frac{\int_\bbR [ - B(x)\psi'(x)^2 + q(x)\psi'(x)\psi(x) + a(x)\psi(x)^2] dx}{\int_\bbR \psi(x)^2 dx} \quad(\ge a_-).
\eeq
Note that when $q\equiv 0$, then the Rayleigh quotient formula for self-adjoint operators gives 
\[
\lambda_0 = \sup \sigma \left[ \partial_x(B(x)\partial_{x})+a(x) \right] .
\]
 As we show below, for $\laa> \lambda_0$ there is again a unique $\phi_\laa>0$ such that 
\beq \lb{1.12}
(B(x)\phi_\laa')' + q(x)\phi_\laa' +a(x)\phi_\laa=\laa\phi_\laa,
\eeq
$\lim_{x\to\infty} \phi_\laa(x)=0$ and $\phi_\laa(0)=1$.

\begin{theorem} \lb{T.1.1}
Assume \eqref{1.2}--\eqref{1.4} and \eqref{1.11}, let $ \lambda_0$ 
be as in \eqref{1.11a} and for $\laa> \lambda_0$ let $\phi_\laa$ be as in \eqref{1.12}.
Let $(aB)_-\equiv  \inf_{x\in\bbR} [a(x)B(x)]$, and assume also that $q_+\le 2\sqrt{(aB)_-}$ and
\beq \lb{1.13}
 \lambda_0< \lambda_1\equiv \inf_{x\in\bbR} \left\{ a(x) +  \sqrt{(aB)_-} \left[ \sqrt{(aB)_-} - |q(x)| \right] B(x)^{-1} \right\}. 
\eeq
Let $\mu$ be a finite non-negative non-zero Borel measure on $(\lambda_0,\lambda_1)$ with $\mu_0\equiv \inf \supp(\mu)$ and $\mu_1\equiv \sup \supp(\mu)$, and define $v_\mu$ as in \eqref{1.8b}.

(i) If $\mu_1<\lambda_1$, then there is an increasing function $h:[0,\infty)\to[0,1)$ with $h(0)=0$, $h'(0)=1$, $\lim_{v\to\infty} h(v)=1$, and an entire solution $u_\mu$  of \eqref{1.9} satisfying \eqref{1.6}, $(u_\mu)_t>0$, 
\beq \lb{1.14}
h \left( v_\mu \right) \le u_\mu \le \min \left\{ v_\mu,1 \right\}.
\eeq
In fact, we can  choose $h= h_{g,\alpha}$ from \eqref{2.1} below, with any $\alpha\in ( 1-(\lambda_1-\mu_1)a_+^{-1},1)$.


(ii) If $\lambda_0<\mu_0\le \mu_1<\lambda_1$, then $u_\mu$ from (i) is a transition front (i.e., satisfying also \eqref{1.7}), with $L_\eps$ depending only on $g,a_+,B_\pm,\eps$ and $\zeta$, provided $\min\{\mu_0-\laa_0, \laa_1-\mu_1\} \ge\zeta>0$.
\end{theorem}

{\it Remarks.} 1.  Condition \eqref{1.13} is sharp in this generality, as exhibited by the previously mentioned  non-existence of transition fronts in the case of $B\equiv 1$, $q\equiv 0$, and compactly supported $a(x)-a_-$ with $a(x)\ge a_->0$ and $\lambda_0>2a_-$ \cite{NRRZ}.
\smallskip

2.  The properties of $h$ give $\lim_{x\to\infty} u_\mu(t,x) v_\mu(t,x)^{-1}=1$ for each $t\in\bbR$.  
\smallskip


3. Note that  $a_- +   \sqrt{(aB)_-} [ \sqrt{(aB)_-} - q_+ ] B_+^{-1} \le \laa_1 \le 2a_-$, so \eqref{1.13} is satisfied when $ \lambda_0< a_- +   \sqrt{(aB)_-} [ \sqrt{(aB)_-} - q_+ ] B_+^{-1}$. 
In the case $B\equiv 1$ and $q\equiv 0$ we have $\lambda_1=2a_-$, so \eqref{1.13} simplifies to $ \lambda_0<2a_-$, the condition mentioned above. 
\smallskip

4. Of course, an identical result holds for solutions moving to the left, with $\psi_\laa$ defined as $\phi_\laa$ but satisfying instead $\lim_{x\to -\infty} \psi_\laa(x)=0$.  In addition, a combination of two solutions of \eqref{1.10} from (i), moving in opposite directions, gives an entire solution of \eqref{1.9} whose spatial infimum converges to 1 as $t\to\infty$.
\smallskip

5. The borderline case $\mu=\delta_{\lambda_1}$, which corresponds to the traveling front with the minimal speed $c_f^*$ and maximal decay $\sim e^{-\sqrt{f'(0)}x}$ when $f(x,u)=f(u)$, is not covered by our result (because then $\alpha=1$ in Lemma \ref{L.2.1} below).  It is an open question whether  a {\it maximal decay transition front} exists in the inhomogeneous setting.
\smallskip

6.  The nonlinearity $f$ can in addition depend on time, as long as $f_u(t,x,0)$ is time independent.  This is also the case for the other results in this paper.
\smallskip

7. Finally, we note that all our results continue to hold if in \eqref{1.2} one does not necessarily require $f(x,1)=0$.  In that case we drop the lower bound on $f$ in \eqref{1.2} for $u>1$, consider solutions $u\ge 0$ (rather than $0\le u\le 1$) not necessarily converging to 1 as $x\to-\infty$, and the upper bound in \eqref{1.14} becomes just $u_\mu(t,x) \le  v_\mu(t,x)$.
\smallskip

Although the ``extremal'' fronts $v_{\delta_\laa}(=v_\laa)$ have a constant speed in homogeneous media, one cannot expect them to have a constant or even a mean speed in general.  However, if the medium is random and stationary ergodic, they do have (almost surely) a deterministic {\it aymptotic speed}
\beq \lb{1.15}
c\equiv \lim_{|t|\to\infty} \frac{X(t)}{t} >0.
\eeq
with $X(t)$ as in \eqref{1.7a}.

\begin{theorem} \lb{T.1.2}
Consider a probability space $(\Omega,\mathcal{F},\bbP)$ and assume that a measurable  function $p\equiv (a,B,q):\Omega\to L^\infty_{\rm loc}(\bbR)^3$ is Lipschitz in $x$ and satisfies \eqref{1.4} and \eqref{1.11}, uniformly in $\omega\in\Omega$. In addition, assume that $p$  is  stationary ergodic. That is,  there is a group  $\{\pi_y\}_{y\in \bbR}$ of measure preserving transformations acting ergodically on $\Omega$ such that  $p({\pi_y\omega};x)=p(\omega;x+y)$.  Then $\lambda_0,\lambda_1$ from Theorem \ref{T.1.1} are constant in $\omega$, except on a measure zero set. If $\lambda_0<\lambda_1$ and a reaction $f(\omega;x,u)$ satisfies \eqref{1.2}--\eqref{1.3a} for almost all $\omega\in\Omega$, then for each $\lambda\in(\lambda_0,\lambda_1)$ there is $c_\lambda>0$ such that the transition front $u_{\delta_\lambda}({\omega};t,x)$  from Theorem~\ref{T.1.1}(ii) has asymptotic speed $c_\lambda$ in the sense of \eqref{1.15} for almost all $\omega\in\Omega$.
\end{theorem}

{\it Remarks.} 1. Notice that $f$ itself need not be stationary ergodic. 
\smallskip

2. If $B\equiv 1$ and $q\equiv 0$, the condition $\lambda_0<\lambda_1$ again becomes $\lambda_0<2a_-$, which is guaranteed, for instance, when $a_+<2a_-$, regardless of the structure of the randomness.  
\smallskip

3.  It is conceivable that, in general, transition fronts exist {\it almost surely} even if $\lambda_0\ge 2a_-$.  We do not know the answer to this question at this time and pose it as an open problem.
\smallskip

We also provide applications of our method in several spatial dimensions, to the study of solutions of the reaction-diffusion equation
\beq \lb{2.4a}
u_t = \nabla\cdot(B(x)\nabla u) + q(x)\cdot\nabla u + f(x,u)
\eeq
on $\bbR\times\bbR^d$, where $f,B,q$ are again as above but with $B$ a matrix field and $q$ a vector field.

Let us start with the special case 
\beq \lb{1.16}
u_t=\Delta u + f(x,u)
\eeq
with $f_u(x,0)\equiv\aaa >0$ independent of $x$.
The corresponding linear PDE
\beq \lb{1.17}
v_t=\Delta v + \aaa v
\eeq
has ``extremal'' solutions $v_{0} (t,x)\equiv e^{\aaa t}$ and 
\[
v_{\gamma\eta} (t,x)\equiv e^{ -\gamma\eta\cdot x + (\gamma^2+\aaa )t } = e^{-\gamma(x\cdot\eta - c_{\aaa ,\gamma} t)},
\]
with $\gamma>0$, $\eta\in\bbR^d$ a unit vector, and as before,
\[
c_{\aaa ,\gamma} = \gamma+ \aaa \gamma^{-1} \ge 2\sqrt{\aaa }. 
\]    

From the one-dimensional case mentioned above it immediately follows that each traveling front for \eqref{1.16} of the form $u(t,x)=U(x\cdot\eta-ct)$ has the same decay (as $x\cdot\eta\to\infty$) as a multiple of $v_{\gamma\eta}$  for some $\gamma\in(0, \sqrt{\aaa }]$ (with an extra factor  $x\cdot\eta - 2\sqrt{\aaa}\, t$  if $\gamma=\sqrt{\aaa }\,$), and then $c=c_{\aaa ,\gamma}$.
Both  $u$ and $v_{\gamma\eta}$ travel with speed $c_{\aaa ,\gamma}$ in the direction $\eta$.

We will therefore only consider $\gamma\le\sqrt{\aaa }$ and let $Y\equiv \overline{B}(0,\sqrt{\aaa })$ be the closed ball in $\bbR^d$ with radius $\sqrt{\aaa }$ and centered at 0, with topology inherited from $\bbR^d$.  If $\mu$ is a finite non-negative non-zero Borel measure  on $Y$, then we let
\beq\lb{1.17a}
v_\mu(t,x)\equiv  \int_Y v_{\xi}(t,x) d\mu(\xi) = \int_Y e^{ -\xi\cdot x + (|\xi|^2+\aaa )t } d\mu(\xi)
\eeq
(i.e.,  $v_{\delta_\xi}=v_\xi$).  Notice that $v_\mu(t,x)\le e^{\sqrt{\aaa }|x|+ \aaa(3+\sgn(t)) t/2}$ and it is a positive entire solution of \eqref{1.17}. 
 Also, $Y$ becomes an analog of $[-\lambda_0,-\lambda_1]\cup [\lambda_0,\lambda_1]$ in Theorem \ref{T.1.1} (the latter set supports measures corresponding to solutions from Remark 4 after Theorem \ref{T.1.1}), after recalling that for homogeneous reactions, $\laa_0=\aaa $, $\laa_1=2\aaa $, and $\gamma=\sqrt{\laa - \aaa }$.
 
Part (i) of our next result shows that each $v_\mu$ gives rise to an entire solution $u_\mu$ of \eqref{1.16}. Moreover, in parts (ii) and (iii) we address the questions when this solution connects 0 and 1 and when does the transition zone between $\eps$ and $1-\eps$ have a bounded width (in some sense) for each $\eps>0$.  
To this end, let  us define the {\it convex hull} of a measure $\mu$ on $\bbR^d$ to be
\[
\ch(\mu) \equiv \{ \zeta\in\bbR^d \,|\, \zeta=\bbE(\nu) \text{ for some measure } 0<\nu\le \mu  \},
\]
with $\bbE(\nu)\equiv \nu(\bbR^d)^{-1}\int_{\bbR^d} \xi d\nu(\xi)$.  Then $\ch(\mu)$  is convex because
\[
\bbE(\beta \nu + (1-\beta)\nu') = [ \beta \nu(\bbR^d) + (1-\beta) \nu'(\bbR^d) ]^{-1} \left[ \beta \nu(\bbR^d) \bbE(\nu) + (1-\beta) \nu'(\bbR^d) \bbE(\nu') \right]
\]
but not necessarily closed.  We note that $\ch(\mu)$ is also  the intersection of convex hulls of all essential supports of $\mu$, that is,  sets $A\subset\bbR^d$ such that $\mu(A)=\mu(\bbR^d)$ and $\mu(A')<\mu(A)$ whenever $A'\subset A$ and $A\setminus A'$ has a positive Lebesgue measure (see the remark after the proof of Theorem \ref{T.1.3}), although $\ch(\mu)$ itself need not be an essential support of $\mu$ (e.g., if $B\subset\bbR^d$ is an open ball and $\mu$ the uniform measure on the sphere $\partial B$, then $\ch(\mu)=B$).

\begin{theorem} \lb{T.1.3}
Assume \eqref{1.2}--\eqref{1.3a} for $x\in\bbR^d$ and with $a(x)\equiv a>0$.  Let $\mu$ be a finite non-negative non-zero Borel measure with support in the open ball $B(0,\sqrt{\aaa })$ and let $v_\mu$ be as in \eqref{1.17a}.

(i) There is an increasing function $h:[0,\infty)\to[0,1)$ with $h(0)=0$, $h'(0)=1$ and $\lim_{v\to\infty}h(v)=1$, and an entire solution $u_\mu$ of \eqref{1.16} such that $(u_\mu)_t>0$ and \eqref{1.14} holds.
In fact, we can  choose $h= h_{g,\alpha}$ from \eqref{2.1} below, 
provided $\mu$ is supported in $\overline{B}(0,\sqrt{\alpha \aaa })$. 
Also, $u_\mu \not\equiv u_{\mu'}$ when $\mu\neq\mu'$. 

(ii) We have
\beq \lb{1.19a}
\inf_{x\in\bbR^d} u_\mu(x,t) = 0 \qquad\text{and}\qquad \sup_{x\in\bbR^d} u_\mu(x,t) = 1
\eeq
for each $t\in\bbR$ (equivalently, for some $t\in\bbR$) if and only if $0\notin \ch(\mu)$.

(iii) If $0\notin \supp (\mu)$, then for each $\eps,\tht>0$ there is $L_{\eps,\tht}$ (depending also on $\dist(0,\supp(\mu))$, $f$, and $\alpha$ from (i)), such that the following holds.   If $u_\mu(t,x)\ge \eps$, then there is a unit vector $\eta_{t,x}\in\bbR^d$ such that $u_\mu(t,x+y)\ge 1-\eps$ whenever $\eta_{t,x}\cdot y|y|^{-1} \ge \tht$ and $|y|\ge L_{\eps,\tht}$.
\end{theorem}


Part (i) of this result is  closely related to a result of Hamel-Nadirashvili \cite[Theorem~1.2]{HN2}.  Under the additional assumptions of $f$ being independent of $x$, concave in $u$, and \hbox{$f\in C^2([0,1])$,} they prove the existence of an infinite-dimensional manifold of entire solutions of \eqref{1.16}.  These solutions are parametrized by measures supported on the 1-point compactification $X$ of \hbox{$\bbR^d\setminus B(0,2\sqrt{\aaa })$}, where distance from origin denotes the front speed $c\ge 2\sqrt{\aaa }$ rather than $\gamma\le \sqrt{\aaa }$.  The mapping $\gamma\mapsto c_{\aaa ,\gamma}$ yields a natural identification of $Y$ and $X$ (we consider the former a slightly more natural parameter space for our method than the latter), so one could ask what is the relationship of the two sets of entire solutions.  

Under the above additional assumptions on $f$, it is also shown in \cite[Theorem 1.4]{HN2} that any entire solution $0<u<1$ which satisfies
\beq \lb{1.20}
\lim_{t\to -\infty} \sup_{|x|<(2\sqrt{\aaa }+\eps)|t|}u(t,x) = 0
\eeq
for some $\eps>0$, is from their manifold.  This gives a characterization of all entire solutions satisfying \eqref{1.20}.  Our $u_\mu$ satisfies \eqref{1.20} with some $\eps(\alpha)>0$ as well as the properties of the solution from \cite[Theorem 1.2]{HN2} corresponding to the measure obtained from $\mu$ under the above-mentioned identification of $Y$ and $X$.  Since these 
properties uniquely define a solution in the manifold, it follows that for $f\in C^2([0,1])$, independent of $x$, and concave in $u$, the two solutions coincide; and the solutions from Theorem \ref{T.1.3}(i) are all the entire solutions of \eqref{1.16} satisfying \eqref{1.20}. 

Moreover, the manifold in \cite[Theorem 1.2]{HN2} also contains solutions corresponding to some measures supported in $X$ but not in its interior (which we do not construct in Theorem~\ref{T.1.3}), namely, those whose restriction to $\partial B(0,2\sqrt{\aaa })$ is a finite sum of Dirac masses.  

However, besides proving the existence of this manifold of solutions, \cite{HN2} only obtains certain claims about the $t\to -\infty$ asymptotic behavior of each of them, with better control only for those corresponding to measures $\mu$ which are finite sums of Dirac masses \cite[Theorem~1.1]{HN2}. The contribution of Theorem \ref{T.1.3}(i) is therefore not only in proving the existence of these entire solutions for more general (and even inhomogeneous) KPP reactions, but also in obtaining the explicit estimate \eqref{1.14}, valid for all times and yielding the new results in (ii) and (iii). Moreover,  the usage of our method (from Lemma \ref{L.2.1} below) makes the proof immediate and elementary, while the proof of \cite[Theorem~1.2]{HN2} is 30 pages long.

In fact, Theorem \ref{T.1.3} extends to some periodic $(a,B,q)$ ($f$ need not be periodic in $x$ and can even be time-dependent, as mentioned above).  Now
\[
v_{\xi} (t,x)\equiv e^{ -\xi\cdot x + \kappa_\xi t } \tht_\xi(x),
\]
where $(\tht_\xi,\kappa_\xi)$ is the unique solution of
\beq \lb{2.4e}
\nabla\cdot (B(x)\nabla\tht) + (q(x) - 2 B(x)\xi) \cdot\nabla\tht + [\xi\cdot B(x)\xi  - \nabla \cdot (B(x)\xi) - q(x)\cdot\xi + a(x)] \tht = \kappa \tht
\eeq
on the unit cell of periodicity $\calC$ (satisfying periodic boundary conditions) with $\tht_\xi>0$ and $\int_{T^{d}}\tht_\xi(x)dx=1$.  Again
\beq\lb{2.4d}
v_\mu(t,x)\equiv  \int_Y v_{\xi}(t,x) d\mu(\xi) 
\eeq
solves
\[
v_t = \nabla\cdot(B(x)\nabla v) + q(x)\cdot\nabla v + a(x)v
\]
when $\mu$ is as above.  Finally, let $S_\alpha$ be the set 
of all $\xi\in\bbR^d$ such that 
\beq\lb{2.4f}
\left\| \left( \frac{\nabla\tht_\xi}{\tht_\xi} -\xi \right) \cdot \frac Ba \left( \frac{\nabla\tht_\xi}{\tht_\xi} -\xi \right) \right\|_{L^\infty(\calC)} \le \alpha.
\eeq

\begin{theorem} \lb{T.1.5}
Assume \eqref{1.2}--\eqref{1.3a} for $x\in\bbR^d$ and with $(a,B,q)$ periodic.  Let $\mu$ be a finite non-negative non-zero Borel measure supported on $S_\alpha$ for some $\alpha<1$, and let $v_\mu$ be as in \eqref{2.4d}.  
Then  Theorem \ref{T.1.3}(i)--(iii) hold with $h=h_{g,\alpha}$ from \eqref{2.1} below, except possibly the last statement in (i).
\end{theorem}

{\it Remark.}  We note that in general, all $S_\alpha$ for $\alpha<1$ may be empty.  However this is not the case when $B-I$ is small in $C^{1,\delta}(\bbT^{d})$ and $a-\bar a,q$ (with $\bar a\equiv \int_{\bbT^{d}} a(x)dx$) are small in $C^\delta(\bbT^{d})$ for some $\delta>0$.  Indeed, in that case we obtain a uniform (in norms of $B-I,a-\bar a,q$ in the respective spaces) bound on $\theta_\xi$ in $C^{2,\delta}(\bbT^d)$ for all $|\xi|\le 1$.  If now $(a-\bar a,B-I,q)\in C^{1,\delta}\times C^{\delta}\times C^{\delta}$ is small enough, then $\kappa_\xi-|\xi|^2-\bar a$ is also small, so  $a(x)+|\xi|^2-\kappa_\xi$ is small in $C^\delta$ and \eqref{2.4e} can be rewritten as
\begin{align*}
\Delta\tht_\xi   +2\xi\cdot\nabla\tht_\xi 
  = & - \nabla\cdot [(B(x)-I)\nabla\tht_\xi] - [q(x) - 2 (B(x)-I)\xi] \cdot\nabla\tht_\xi  \\
 & - [\xi\cdot (B(x)-I)\xi  - \nabla \cdot (B(x)\xi) - q(x)\cdot\xi + a(x)+|\xi|^2-\kappa_\xi] \tht_\xi ,
\end{align*}
with the right-hand side uniformly small in $C^\delta$ for all $|\xi|\le 1$.  Thus $\tht_\xi- \int_{T^{d-1}}\tht_\xi(x)dx = \tht_\xi-1$ is uniformly small in $C^{2,\delta}$.  This means that for each $\beta<1$, \eqref{2.4f} holds for  $\alpha\equiv \tfrac 12(1+\beta)$ and all $|\xi|\le\beta$ provided $(a-\bar a,B-I,q)$ is sufficiently small in $C^{1,\delta}\times C^{\delta}\times C^{\delta}$.
\smallskip

We end this introduction  with an application of our method to obtaining explicit bounds on certain solutions $u$ of \eqref{1.16} with constant $f_u(x,u)=a$, in terms of the solutions of the {\it heat equation} $\til u_t=\Delta \til u$ with the same initial condition (in which case $\til u\le u \le e^{\aaa t}\til u $).  Of course, the latter is just
\beq\lb{1.20a}
\til u(t,x)=(4\pi t)^{-d/2} \int_{\bbR^d} e^{-|x-y|^2/4t} u(0,y) dy.
\eeq

\begin{theorem} \lb{T.1.4}
Assume \eqref{1.2}--\eqref{1.3a} for $x\in\bbR^d$ and with $a(x)\equiv a>0$.  Let $0\le u\le 1$ solve \eqref{1.16} on $\bbR^+\times\bbR^d$.  If $\til u$ from \eqref{1.20a} satisfies
\beq \lb{1.21}
|\nabla \til u(t_0,x)| \le \sqrt{\alpha \aaa }\, \til u(t_0,x)
\eeq 
for some $t_0\ge 0$, $\alpha<1$, and all $x\in\bbR^d$, then
\beq \lb{1.22}
h_{g,\alpha} \left( e^{\aaa (t-t_0)} \til u(t,x) \right) \le u(t,x) \le \min\{e^{\aaa t} \til u(t,x), 1\}
\eeq 
for all $(t,x)\in \bbR^+\times\bbR^d$, with $h_{g,\alpha}$ from \eqref{2.1} below  (in particular,  $h_{g,\alpha}'(0)=1=h_{g,\alpha}(\infty)$).
\end{theorem}

We prove Theorems \ref{T.1.1}--\ref{T.1.4} in the next section, after introducing our main tool, Lemma~\ref{L.2.1}.

Finally, we note that existence of transition fronts for \eqref{1.1} with very general $f$ (including KPP) is claimed in the paper \cite{SLL}.  This statement is false in the full generality claimed there (in particular, it contradicts the non-existence result in \cite{NRRZ}), and its proof is also incorrect.  The latter is a direct adaptation of the existence-of-fronts proof for ignition reactions from \cite{MRS} which, however, does not extend to non-ignition reactions.  In particular, various claims in \cite{SLL}, such as the one between (2.22) and (2.23), Corollary 2.6(i), and Proposition 2.7, are made without a proof and are, in fact, false for general  non-ignition reactions.

The author would like to thank Fran\c cois Hamel for pointing out the argument in the remark after Theorem \ref{T.1.5}.  He also acknowledges partial support by NSF grants DMS-1113017 and DMS-1056327, and by an Alfred P. Sloan Research Fellowship.

\section{The Key Lemma and the Proofs of Theorems \ref{T.1.1}--\ref{T.1.4}} \lb{S2}

Our main tool is the following lemma, which constructs sub-solutions $w=h(v)$ of  \eqref{1.9} from certain solutions $v$ of \eqref{1.10} (which are  also super-solutions of \eqref{1.9}).  Here the function $h=h_{g,\alpha}:[0,\infty)\to[0,1)$ depends on $g\in C^1([0,1])$ satisfying \eqref{1.3}, \eqref{1.3a} and also on an additional parameter $\alpha\le 1$.  Specifically, $h_{g,\alpha}(0)=0$ and
\beq \lb{2.1}
h_{g,\alpha}(v) \equiv  U_{g,\sqrt{\alpha}}(-\alpha^{-1/2} \ln v)
\eeq
for $v>0$, where $U_{g,\sqrt{\alpha}}$ is the traveling front profile for the homogeneous PDE
\beq \lb{A.1}
u_t=u_{xx} + g(u)
\eeq
corresponding to speed $c_{1,\sqrt{\alpha}}\equiv \alpha^{1/2} + \alpha^{-1/2}\ge 2$.    That is, $U_{g,\sqrt{\alpha}}(-\infty)=1$, $U_{g,\sqrt{\alpha}}(\infty)=0$, $U_{g,\sqrt{\alpha}}'<0$, and 
\beq \lb{A.1a}
U''_{g,\sqrt{\alpha}}+c_{1,\sqrt{\alpha}} U'_{g,\sqrt{\alpha}}+g(U_{g,\sqrt{\alpha}})=0
\eeq
on $\bbR$.   Notice that the $\lim_{v\to\infty} h_{g,\alpha}(v)=1$ and \eqref{A.1a} implies
\beq \lb{2.5d}
\alpha v^2 h_{g,\alpha}''(v) - vh_{g,\alpha}'(v) + g(h_{g,\alpha}(v))=0.
\eeq

 It is well known that $U_{g,\sqrt{\alpha}}$ is unique up to translation and if $\alpha<1$, then there is a unique translation such that $\lim_{x\to\infty} U_{g,\sqrt\alpha}(x) e^{\sqrt\alpha x}=1$ \cite{Uchi}.  With this choice of $U_{g,\sqrt\alpha}$ we obtain $h_{g,\alpha}'(0)=1$ for $\alpha<1$.  It then also follows that 
 \beq\lb{2.6c}
 h_{g,\alpha}(v)\le v
 \eeq
 for $v\in[0,\infty)$ because $h_{g,\alpha}''<0$ (see the proof of Lemma \ref{L.2.1} below).
 
 For $\alpha=1$ we instead have $\lim_{v\to 0} h_{g,\alpha}(v)(-v\ln v)^{-1}=1$, provided the first condition in \eqref{1.3a} is replaced by $\int_0^1 [u-g(u)] |\ln u| u^{-2} du<\infty$ \cite{Uchi}.  

We state the lemma in a more general form, with time-dependent coefficients.


\begin{lemma} \lb{L.2.1}
With $f,a,B,q$ Lipschitz and time-dependent ($B$ a matrix and $q$ a vector field), assume \eqref{1.2}--\eqref{1.4} and \eqref{1.11} for $(t,x)\in(t_0,t_1)\times\bbR^d$ (where $a(t,x)\equiv f_u(t,x,u)$ and  \hbox{$-\infty< t_0<t_1\le\infty$}).  Let $v>0$ be a solution of
\[
v_t = \nabla\cdot(B(t,x)\nabla v) + q(t,x)\cdot\nabla v + a(t,x) v
\]
on $(t_0,t_1)\times\bbR^d$.
If  for some $\alpha<1$,
\beq \lb{2.3}
\nabla v(t,x) \cdot B(t,x) \nabla v(t,x) \le \alpha a(t,x) v(t,x)^2
\eeq
holds for all $(t,x)\in (t_0,t_1)\times\bbR^d$, then $v$ and $w\equiv h_{g,\alpha}(v)$ are a super- and sub-solution  of 
\beq \lb{2.4}
u_t = \nabla\cdot(B(t,x)\nabla u) + q(t,x)\cdot\nabla u + f(t,x,u)
\eeq
on $(t_0,t_1)\times\bbR^d$. 
Therefore, if $0\le u\le 1$ solves \eqref{2.4} with $w(t_0,x)\le u(t_0,x)\le v(t_0,x)$ for all $x\in\bbR^d$, then for all $(t,x)\in (t_0,t_1)\times\bbR^d$ we have
\beq \lb{2.5}
w(t,x)\le u(t,x)\le \min\{v(t,x),1\}.
\eeq
\end{lemma}

{\it Remark.} Of course, the crucial hypothesis here is \eqref{2.3}.

\begin{proof}
Obviously $v$ is a super-solution of \eqref{2.4}, giving the second inequality.  We also have
\begin{align*}
w_t -  \nabla\cdot(B\nabla w) - q\cdot\nabla w & = h'(v) [v_t -  \nabla\cdot(B\nabla v) - q\cdot\nabla v] - h''(v) \nabla v \cdot B \nabla v 
\\ & =h'(v) a v - h''(v) \nabla v \cdot B \nabla v 
\\ & \le a [ h'(v) v - \alpha h''(v) v^2 ].
\end{align*}
In the last inequality we used \eqref{2.3} and $h''<0$.  The latter is due to \eqref{2.5d} and Lemma \ref{L.A.1} from the Appendix with $\gamma\equiv \sqrt\alpha$, which yield
\[
\alpha v^2 h''(v) = vh'(v) - g(h(v)) = -\alpha^{-1/2} U_{g,\sqrt\alpha}'(-\alpha^{-1/2} \ln v) -  g(U_{g,\sqrt\alpha}(-\alpha^{-1/2} \ln v))<0.
\]
Thus \eqref{2.5d} and \eqref{1.2} give
\[
w_t -  \nabla\cdot(B(t,x)\nabla w) - q(t,x)\cdot\nabla w \le a(t,x)g(h(v)) \le f(t,x,w),
\]
so $w$ is a sub-solution of \eqref{1.9}, and the first inequality in \eqref{2.5} follows as well.
\end{proof}



\begin{proof}[Proof of Theorem \ref{T.1.4}]
The comparison principle, together with \eqref{1.2} yields the upper bound, as well as $\til u\le u$.  Then let $v(t,x)\equiv e^{\aaa (t-t_0)}\til u(t,x)$ and note that $r\equiv\nabla v v^{-1} = \nabla\til u\til u^{-1}$ satisfies
\[
r_t=\Delta r + \nabla(|r|^2)
\]
because
\[
(\ln \til u)_t = {\Delta \til u} {\til u^{-1}} = \Delta(\ln \til u) + |r|^2.
\]
Thus $\rho\equiv |r|^2$ satisfies
\[
\rho_t=\Delta \rho +2r\cdot \nabla \rho -2|\nabla r|^2,
\]
so \eqref{1.21} and the maximum principle give $\rho(t,x) \le \alpha \aaa $ for $(t,x)\in(t_0,\infty)\times\bbR^d$. Then Lemma~\ref{L.2.1} yields the lower bound in \eqref{1.22}.
\end{proof}

\begin{proof}[Proof of Theorem \ref{T.1.1}]
(i) Let us start with the proof of existence of $\phi_\lambda$ from \eqref{1.12}, for $\lambda>\lambda_0$.  With $\calL$ the operator on the left-hand side of \eqref{1.12} and $\lambda_0$ from \eqref{1.11a},  we have
\[
\int_\bbR \psi(x)[(\lambda-\calL)\psi](x) dx \ge (\lambda-\lambda_0) \int_\bbR \psi(x)^2dx
\]
for $\psi\in H^2(\bbR)$, after integrating by parts.
Thus $(\lambda-\calL)^{-1}:L^2(\bbR)\to H^2(\bbR)$ exists and if $0\not\equiv \psi\in L^2(\bbR)$ is compactly supported in $\bbR^-$, then $0\not\equiv \phi\equiv (\lambda-\calL)^{-1}\psi\in H^2(\bbR)$.  Since $\phi$ also satisfies \eqref{1.12} on $\bbR^+$, Harnack inequality shows that $\lim_{x\to\infty} \phi(x)=0$.  Let $\til \phi(x)\equiv \phi(x)$ for $x\ge 0$ and extend it onto $\bbR^-$ so that it solves \eqref{1.12}.  Then $\til \phi$ has no roots because if $\til\phi(x_0)=0$, then plugging the function $\til\phi|_{[x_0,\infty)}$, extended by 0 on $(-\infty,x_0)$, into \eqref{1.11a} would yield $\lambda_0\ge\lambda$.  Thus we have $\phi_\lambda(x) = \til\phi(x)\til\phi(0)^{-1}$.  Uniqueness follows from existence of $\psi_\lambda$ with the same properties but with $\lim_{x\to -\infty} \psi_\lambda(x)=0$ (by a reflected argument), from $\lim_{x\to -\infty} \phi_\lambda(x)=\infty$ (by \eqref{2.8} below), and the fact that the space of solutions of \eqref{1.12} is two-dimensional.  

Next, choose $\alpha<1$ such that 
\[
m\equiv  \inf_{x\in\bbR, \beta\ge\alpha} \left\{ a(x) + \sqrt{\beta (aB)_-} \left[ \sqrt{\beta (aB)_-} - |q(x)| \right] B(x)^{-1}  \right\} - \mu_1>0.
\]
Any $\alpha\in( 1-(\lambda_1-\mu_1)a_+^{-1},1)$ works because the derivative of the expression in the brackets with respect to $\beta$ is bounded above by $(aB)_-B(x)^{-1}\le a_+$  and is positive for $\beta> 1$ (the latter due to $q_+\le 2\sqrt{(aB)_-}$). Now let $w_\mu(t,x)\equiv h_{g,\alpha}(v_\mu(t,x))$ and notice that $w_\mu\le v_\mu$ by \eqref{2.6c}.  Then Lemma~\ref{L.2.1} will be applicable to $v_\mu,w_\mu$ once we establish
\beq \lb{2.6}
B(x) \phi_\laa'(x)^2 \le \alpha a(x) \phi_\laa(x)^2
\eeq
for all $\laa\in(\lambda_0,\mu_1]$ and $x\in\bbR$.  Indeed, \eqref{2.6} and $\phi_\laa>0$ then yield \eqref{2.3} for $v_\mu$.  To this end, we need to show 
\beq \lb{2.7}
|\psi(x)| \le \sqrt{\alpha a(x)B(x)}
\eeq
for $x\in \bbR$, with $\psi \equiv B\phi_\laa'/\phi_\laa$ and $\laa\in(\lambda_0,\mu_1]$.  

Let us assume that $\psi(x_0)\ge \sqrt{\alpha (aB)_-}$ for some $x_0$.  We have $\psi'=\laa-a-\psi(\psi+q)B^{-1}$ on $\bbR$, so $\psi'(x_0)\le \lambda-m-\mu_1\le -m$.  
But then $\psi$ must be decreasing on $(-\infty,x_0]$ with $\psi'\le-m$ there.
From this and $\psi'=\laa-a-(\psi^2+q\psi)B^{-1}$ it follows that $\psi$ must blow up at some $x_1\in (-\infty, x_0)$, a contradiction.  We obtain the same conclusion when assuming $\psi(x_0)\le - \sqrt{\alpha (aB)_-}$ (because $\psi'=\laa-a-|\psi|(|\psi|-q)B^{-1}$ when $\psi<0$), with blowup at some $x_1\in (x_0,\infty)$.  It follows that $\| \psi \|_\infty\le \sqrt{\alpha(aB)_-}$, which gives \eqref{2.7}, so Lemma \ref{L.2.1} applies to $v_\mu,w_\mu,\alpha$.

A standard limiting argument (see, for instance, \cite{FM}) now recovers an entire solution to \eqref{1.9} between $\min\{v_\mu,1\}$ and $w_\mu$.  Indeed, we let $u_k$ be the solution of \eqref{1.9} on $(-k,\infty)\times\bbR$ with initial datum $u_k(-k,x)\equiv w_\mu(-k,x)$.  Then by Lemma \ref{L.2.1} we have 
\[
w_\mu(t,x) \le u_k(t,x)\le \min\{v_\mu(t,x),1\}
\] 
on $(-k,\infty)\times\bbR$.   By parabolic regularity, there is a locally uniform (on $\bbR^2$) limit $u_\mu\in[w_\mu,\min\{v_\mu,1\}]$ of $u_k$ (along a subsequence if needed), which is an entire solution of \eqref{1.9}. Since $(w_\mu)_t\ge 0$, the same is true for $u_k$ and thus $u_\mu$, by the maximum principle.  The strong maximum principle then gives $(u_\mu)_t>0$ because $(u_\mu)_t\not\equiv 0$. 

Finally, \eqref{1.6} follows from \eqref{1.14} and $v_\mu(-\infty)=\infty$, the latter being due to \eqref{2.8} below.

(ii) The fact that $u_\mu$ is a transition front with a bounded width in the sense of \eqref{1.7} when $\lambda_0<\mu_0\le \mu_1<\lambda_1$ will follow from the existence of $L>0$ such that 
\beq \lb{2.8}
\phi_\lambda(c) \ge 2 \phi_\lambda(d)
\eeq
whenever $\lambda\in[\mu_0,\mu_1]$ and $c\le d-L$.  Indeed, we will show that such $L$ depends only on  $a_+,B_\pm,\zeta$, provided $\mu_0-\lambda_0\ge\zeta>0$.  Then \eqref{2.8} holds with the same $L$ for $v_\mu$ in place of $\phi_\lambda$.  Therefore, if  now $\min\{\mu_0-\laa_0, \laa_1-\mu_1\} \ge\zeta>0$, then this and (i) gives \eqref{1.7} with $L_\eps$ depending only on  $g,a_+,B_\pm,\eps,\zeta$.


We are left with proving \eqref{2.8}. 
If in \eqref{1.11a} we take 
\[
\psi(x)\equiv
\begin{cases}
\phi_\lambda (x) & x\in(c,d), 
\\ \phi_\lambda(c)(x-c+1) & x\in[c-1,c],
\\ \phi_\lambda(d)(d+1-x) & x\in[d,d+1],
\\ 0 & x\in\bbR\setminus [c-1,d+1]
\end{cases}
\]
for some $c<d$, then we obtain using \eqref{2.6} and $\alpha<1$,
\begin{align*}
\int_ \bbR & [- B(x)\psi'(x)^2 + q(x)\psi'(x)\psi(x) + a(x)\psi(x)^2 ] dx   
\\ \ge & \int_c^d [- B(x)\phi_\lambda'(x)^2 + q(x)\phi_\lambda'(x)\phi_\lambda(x) + a(x)\phi_\lambda(x)^2 ] dx - (B_+ + q_+)(\phi_\lambda(c)^2+\phi_\lambda(d)^2)
\\ \ge & \int_c^d [(B(x)\phi_\lambda'(x))' + q(x)\phi_\lambda'(x) + a(x)\phi_\lambda(x)  ]\phi_\lambda(x) dx 
\\ & \qquad - (B_+ + q_+)(|\phi_\lambda'(c)|\phi_\lambda(c) + |\phi_\lambda'(d)|\phi_\lambda(d) + \phi_\lambda(c)^2+\phi_\lambda(d)^2)
\\ \ge &  \lambda \int_c^d \phi_\lambda(x)^2 dx - (B_+ + q_+)(1+a_+^{1/2}B_-^{-1/2})( \phi_\lambda(c)^2+\phi_\lambda(d)^2).
\end{align*}
This and \eqref{1.11a} 
give
\[
\lambda_0 \int_c^d \phi_\lambda(x)^2dx \ge  \lambda \int_c^d \phi_\lambda(x)^2 dx - [\lambda_0 + (B_+ + q_+)(1+a_+^{1/2}B_-^{-1/2})]( \phi_\lambda(c)^2+\phi_\lambda(d)^2),
\]
which after setting $M\equiv [\lambda_0 + (B_+ + q_+)(1+a_+^{1/2}B_-^{-1/2})](\lambda-\lambda_0)^{-1}$ reads
\beq \lb{2.10}
 \int_c^d \phi_\lambda(x)^2dx \le M( \phi_\lambda(c)^2+\phi_\lambda(d)^2).
\eeq

By the Harnack inequality, there is $N>0$ such that $\phi_\lambda(y)\le N \phi_\lambda(x)$ if $|x-y|\le 2M$.  Set $L\equiv 6MN^2$ and assume \eqref{2.8} is violated for some $c\le d-L$ (notice that $L$ depends only on $a_+,B_\pm,\zeta$ if $\mu_0-\lambda_0\ge\zeta>0$, because $\lambda_0\le a_+$ and $q_+\le 2\sqrt{a_+B_+}$).   Then there must be $x\in[c,d]$ such that $\phi_\lambda(x)\le N^{-1}\phi_\lambda(d)$ because otherwise
\[
 \int_c^d \phi_\lambda(x)^2dx \ge 6M\phi_\lambda(d)^2 > M( \phi_\lambda(c)^2+\phi_\lambda(d)^2),
\] 
contradicting \eqref{2.10}.  Let $y$ be the rightmost point such that $y<d$ and $\phi_\lambda(y)= N^{-1}\phi_\lambda(d)$, and $z$ the leftmost point such that $z>d$ and $\phi_\lambda(z)= N^{-1}\phi_\lambda(d)$.  Then $y\le d-2M$, $z\ge d+2M$, and $\phi_\lambda(x)\ge N^{-1}\phi_\lambda(d)$ for any $x\in[y,z]$.  But this contradicts \eqref{2.10} with $y,z$ in place of $c,d$, so \eqref{2.8} is proved and we are done.
\end{proof}

{\it Remark.}  The argument in (i) works even for $\mu_1=\lambda_1$, with $\alpha=1$ and $m=0$.  Then $w_\mu\equiv h_{g,1}(v_\mu)$ will again be a sub-solution of \eqref{1.9} but this time $w_\mu \not\le v_\mu$ so we cannot recover a solution between them.

\begin{proof}[Proof of Theorem \ref{T.1.2}]
From \eqref{1.11a} 
we know that $\lambda_0:L^\infty_{\rm loc}(\bbR)^3\to\bbR$ is lower semi-continuous, which together with measurability of $p:\Omega\to L^\infty_{\rm loc}(\bbR)^3$ means that $A_\zeta\equiv \{\omega\in\Omega\,|\, \lambda_0(\omega)>\zeta\}$ is a measurable set.  Obviously $\pi_y A_\zeta=A_\zeta$ for all $y\in\bbR$, so $\bbP(A_\zeta)\in\{0,1\}$ for each $\zeta\in\bbR$.  This means that $\lambda_0$ is almost constant on $\Omega$.  The same follows for $\lambda_1$, using its upper semi-continuity as a function on $L^\infty_{\rm loc}(\bbR)^3$, which follows from its definition.

Let us replace $\Omega$ by its full-measure subset on which $\lambda_0,\lambda_1$ are constant.  Next fix any $\lambda\in(\lambda_0,\lambda_1)$ and let $u_{\delta_\lambda}(\omega;t,x)$ be the corresponding random transition front.  The remark after the proof of Theorem \ref{T.1.1} shows that there is $L$ such that   \eqref{2.8} holds for any $\omega\in\Omega$ and $c\le d-L$.  Therefore also $L_\eps$ in that proof is uniform in $\omega$, which means that if $Y(\omega;t)$  is the rightmost point such that $e^{\lambda t}\phi_\lambda(\omega;Y(\omega;t))=\tfrac 12$ and $X(\omega;t)$ the rightmost point such that $u_{\delta_\lambda}(\omega; t, X(\omega;t))=\tfrac 12$, then $|X(\omega;t)-Y(\omega;t)|$ is uniformly bounded on $\Omega\times\bbR$.  Thus we only need to prove \eqref{1.15} for $Y$ in place of $X$.

Notice that if $r_\lambda(\omega)\equiv \phi_\lambda'(0)$, then $r_\lambda:\Omega\to\bbR$ is measurable because $p:\Omega\to p(\Omega)$ is measurable and $r_\lambda:p(\Omega)\to \bbR$ is continuous when $p(\Omega)$ is equipped with $L^\infty_{\rm loc}(\bbR)^3$-induced topology.  The latter follows from \eqref{2.8} and the fact that any solution of \eqref{1.12} with $\phi(0)=1$ and $\phi'(0)\neq r_\lambda(\omega)$  grows exponentially as $x\to\infty$ (by \eqref{2.8} applied to the solution $\psi_\lambda$ converging to 0 as $x\to-\infty$ and the fact that $\phi_\lambda,\psi_\lambda$ are a basis of the set of all solutions). 

Therefore $\phi_\lambda(\cdot;x)$ is measurable for any fixed $x$.  Since $\phi_\lambda(\pi_y\omega;\cdot)=\phi_\lambda(\omega;y)^{-1}\phi_\lambda(\omega;y+\cdot)$, we have $\phi_\lambda(\omega;y+x)=\phi_\lambda(\omega;y)\phi_\lambda(\pi_y\omega;x)$.  So from ergodicity of $\{\pi_y\}_{y\in\bbR}$ and Oseledec theorem it follows that for almost all $\omega\in\Omega$,
\[
\lim_{x\to \pm\infty} \frac 1x \ln \phi_\lambda(\omega; x) = -\tau_\pm
\]
for some $\tau_\pm\in\bbR$ (and $\tau_\pm>0$ by \eqref{2.8}).  Moreover, $\tau_+=\tau_-$.  Otherwise, 
there exists $\Omega'\subset\Omega$ and $M<\infty$ such that $\bbP(\Omega')>\tfrac 12$ and 
\[
\left |\frac 1{\pm M} \ln \phi_\lambda(\omega; {\pm M}) -\tau_\pm \right| < \frac{|\tau_+-\tau_-|}2
\]
for all $\omega\in\Omega'$. But then
\[
\left |\frac 1M \ln \phi_\lambda(\pi_{-M} \omega; M) -\tau_- \right| < \frac{|\tau_+-\tau_-|}2
\]
for all $\omega\in\Omega'$, so $\Omega'\cap\pi_{-M}\Omega'=\emptyset$, a contradiction with $\bbP(\pi_{-M}\Omega')=\bbP(\Omega')>\tfrac 12$.
Then $\tau_+=\tau_-$ and \eqref{2.8} give 
\[
\lim_{|t|\to\infty} \frac{Y(\omega; t)}t = \frac \lambda{\tau_\pm} \equiv c_\lambda
\]
and the result follows.
\end{proof}


\begin{proof}[Proof of Theorem \ref{T.1.3}]
(i)  The proof of all the claims, with the exception of the last one, is identical to the proof of Theorem \ref{T.1.1}(i), with $\alpha<1$ from the statement of Theorem \ref{T.1.3}(i), and \eqref{2.6} replaced by 
\[
|\nabla v_\xi(t,x)|^2 = |\xi|^2 v_\xi(t,x)^2 \le\alpha \aaa  v_\xi(t,x)^2
\]
for all $|\xi|\le \sqrt{\alpha \aaa }$.

The last claim is an easy consequence of $u_\mu(t,x)v_\mu(t,x)^{-1} \to 1$ as $v_\mu(t,x)\to 0$ and of
\[
\left( \frac{|t|}\pi \right)^{d/2} v_\mu(t,2t\zeta) e^{(|\zeta|^2-\aaa )t} d\zeta \rightharpoonup d\mu(\zeta)
\]
as $t\to -\infty$.
The latter statement, similar to one in \cite{HN2}, follows from
 \[
\left( \frac{|t|}\pi \right)^{d/2} v_\mu(t,2t\zeta) e^{(|\zeta|^2-\aaa )t} = \int_Y  \left( \frac{|t|}\pi \right)^{d/2} e^{-|\xi-\zeta|^2|t|} d\mu(\xi)
 \]
 for $\zeta\in\bbR^d$ and $t<0$.  
 
(iii) If $u_\mu(t,x)\ge \eps$, then $v_\mu(t,x)\ge h^{(-1)}(\eps)$ with $h$ from (i).  Then there is a unit vector $\eta=\eta_{x,t}$ such that
\[
\int_{Y_{\eta,\tht}} e^{-\xi\cdot x+(|\xi|^2+\aaa )t} d\mu(\xi) \ge \frac \tht{2\pi} h^{(-1)}(\eps), 
\]
where
\[
Y_{\eta,\tht} \equiv \left\{\xi\in Y \,\bigg|\, \arccos \frac {-\eta\cdot\xi} {|\xi|} \le \frac\tht 2  \right\}.
\]
If now $\eta\cdot y|y|^{-1} \ge \tht$, then $\arccos(\eta\cdot y|y|^{-1}) \le \tfrac \pi 2-\tht$, and so $\arccos(-\xi\cdot y|y|^{-1}|\xi|^{-1}) \le \tfrac {\pi-\tht} 2$ for any $\xi\in Y_{\eta,\tht}$. Therefore
\[
v_\mu(t,x +y) \ge \int_{Y_{\eta,\tht}} e^{-\xi\cdot (x+y)+(|\xi|^2+\aaa )t} d\mu(\xi) \ge  \frac \tht{2\pi} h^{(-1)}(\eps)   |y| \dist(0,\supp(\mu))  \cos \frac {\pi-\tht} 2
\]
and the result follows from \eqref{1.14} with 
\[
L_{\eps,\tht}\equiv \left[\frac \tht{2\pi} h^{(-1)}(\eps)  \dist(0,\supp(\mu)) \cos \frac {\pi-\tht} 2 \right]^{-1} h^{(-1)}(1-\eps).
\]
 
(ii) Assume first that $0\in\ch(\mu)$ and $\nu(Y)^{-1}\int_{Y} \xi d\nu(\xi) = 0$ for some $0<\nu\le\mu$. Then
\[
v_\mu(t,x) \ge  \int_{Y} e^{-\xi\cdot x + \aaa(3-\sgn(t)) t/2} d\nu(\xi) \ge \nu(Y) e^{- \nu(Y)^{-1}\int_{Y} \xi d\nu(\xi)\cdot x} e^{\aaa(3-\sgn(t)) t/2} = \nu(Y) e^{\aaa(3-\sgn(t)) t/2}
\]
by Jensen's inequality.  This and \eqref{1.14} yield $\inf_{x\in\bbR^d} u_\mu(x,t) >0$ for each $t\in\bbR$.

Now assume that $0\notin\ch(\mu)$ and define $\hat\mu_d\equiv \mu$.  The second claim in \eqref{1.19a} follows from $\mu>0$ and (i) so let us prove the first claim.  Since $\ch(\mu)$ is a convex set, it must be contained in a closed half-space with 0 on its boundary.  Assume without loss it is $\bbR^{d-1}\times\bbR^+_0$, and let $\mu_d\equiv \hat\mu_d|_{\bbR^{d-1}\times\bbR^+}$ and $\hat \mu_{d-1}\equiv \hat\mu_d|_{\bbR^{d-1}\times\{0\}} = \hat\mu_d-\mu_d$.  Now $\ch(\mu)\cap (\bbR^{d-1}\times\{0\})$ must be contained in a closed half-space of $\bbR^{d-1}\times\{0\}$ with 0 on its boundary.  Assume without loss it is $\bbR^{d-2}\times\bbR^+_0\times\{0\}$, and let $\mu_{d-1}\equiv \hat\mu_{d-1}|_{\bbR^{d-2}\times\bbR^+\times\{0\}}$ and $\hat \mu_{d-2}\equiv \hat\mu_{d-1}-\mu_{d-1}$.  Continue in this way until obtaining $\mu_1 = \hat\mu_1$ supported in $\bbR^+\times\{0\}^{d-1}$ (because $\hat\mu_0=\mu|_{\{0\}}=0$).

Since $\mu=\mu_1+\dots+\mu_d$ and $u_\mu\le v_\mu$, it is sufficient to show that for  any $\eps>0$ there is $x\in\bbR^d$ such that for  $k=1,\dots,d$ we have
\beq \lb{2.11}
\int_{Y} e^{-\xi\cdot x} d\mu_k(\xi) \le \frac \eps d
\eeq
(the extra factor $e^{(|\xi|^2+\aaa )t}\le e^{\aaa(3+\sgn(t)) t/2}$ from the definition of $v_\mu$ can be absorbed in $\eps$). For $k=1$, the set of $x\in\bbR^d$ satisfying \eqref{2.11} contains some half-space $[\rho_1,\infty)\times\bbR^{d-1}$.  For each $k= 2,\dots ,d$ and any $r_k>0$, it contains $\bar B_{r_k}(0)\times[\rho_{k,r_k},\infty)\times\bbR^{d-k}$ for some $\rho_{k,r_k}>0$, where $\bar B_{r_k}(0)$ is the closed  ball in $\bbR^{k-1}$ with radius $r_k$ and center 0.  If we choose $r_2\ge \rho_1$ and then recursively $r_k\ge r_{k-1}+\rho_{k-1,r_{k-1}}$ for $k=3,\dots,d$, the corresponding $k$ sets 
all contain the point $x=(\rho_1,\rho_{2,r_2},\dots,\rho_{d,r_d})$.  So \eqref{2.11} holds for this $x$ and we are done.
\end{proof}

{\it Remark.}  We have $\ch(\mu)\subseteq\ches(\mu)$, the intersection of convex hulls of all essential supports of $\mu$.  This is because if  $A$ is an essential support of $\mu$ and $\ch(A)$ its convex hull, then $\bbE(\nu)=\nu(\bbR^d)^{-1}\int_A\xi d\nu(\xi)\in\ch(A)$ when $0<\nu\le\mu$.  The opposite inclusion follows from the construction at the end of the previous proof applied to any $\zeta\notin\ch(\mu)$ instead of 0. Indeed, for any such $\zeta$, one can again find open half-spaces $S_d,\dots,S_1$ of dimensions $d,\dots,1$ whose boundaries contain $\zeta$ (without loss these can be assumed to be $S_k=\zeta + \bbR^{k-1}\times\bbR^+\times\{0\}^{d-k} $) and measures $\mu_k$  on $S_k$ ($k=d,\dots,1$) such that $\mu=\mu_1+\dots+\mu_d$. Thus $S\equiv \bigcup_{k=1}^d S_k$ is an essential support of $\mu$ and $\zeta\notin S$, which yields $\ch(\mu)\supseteq\ches(\mu)$.

\begin{proof}[Proof of Theorem \ref{T.1.5}]
This is identical to the previous proof, using that \eqref{2.4f} yields \eqref{2.3} for $v_\xi$ when $\xi\in S_\alpha$, and thus also for $v_\mu$ because $v_\xi>0$.
\end{proof}

\medskip
\appendix
\section*{Appendix.} \lb{SA}
\renewcommand{\theequation}{3.\arabic{equation}}
\renewcommand{\thetheorem}{3.\arabic{theorem}}
\setcounter{theorem}{0}
\setcounter{equation}{0}

\begin{lemma} \lb{L.A.1}
Assume that $g\in C^1([0,1])$ satisfies \eqref{1.3} and $g'(u)\le 1$ for $u\in(0,1)$.  Let $U:\bbR\to(0,1)$ be a traveling front profile for \eqref{A.1}
corresponding to speed $\gamma+\gamma^{-1}\ge 2$ with $\gamma\in(0,1]$, that is, $U(-\infty)=1$, $U(\infty)=0$, $U'(x)<0$ for all $x\in\bbR$, and  $U$ satisfies
\[
U''+(\gamma+\gamma^{-1})U'+g(U)=0
\]
on $\bbR$.  Then 
\[
0<-U'< \gamma g(U).
\]
\end{lemma}

\begin{proof}
Let $V\equiv U'$ and consider the curve $\{(U(x),V(x))\}_{x\in\bbR}$ in $\bbR^2$.  It connects $(1,0)$ to $(0,0)$ and lies in the fourth quadrant $U>0>V$.  We need to show that it lies in the domain 
\[
D\equiv \{(u,v) \,|\, u\in(0,1) \text{ and } v\in(-\gamma g(u),0) \}.
\]
We have $(U',V')=(V,-\gamma V - \gamma^{-1} V -g(U))$ and the condition $g'\le 1$ ensures that the vector  $(v,-\gamma v - \gamma^{-1} v -g(u))$ points inside $D$ (or is parallel to $\partial D$) when $v=-\gamma g(u)$.  This means that $(U(y),V(y))\in D$ for all $y\ge x$ whenever $(U(x),V(x))\in D$.  Thus if $(U(x),V(x))\notin D$ for some $x\in\bbR$, then $(U(y),V(y))\notin D$ for all $y\le x$.  But then $V(y)\le -\gamma g(U(y))$ for $y\le x$, so $-\gamma V(y) - \gamma^{-1} V(y) -g(U(y))\ge -\gamma V(y)> 0$ for $y\le x$.  Since $V(-\infty)=0$, it follows that $V(x)> 0$, a contradiction.
\end{proof}



\end{document}